\documentclass{amsart}

\usepackage{fullpage}

\newcommand{\lgra}{\longrightarrow}

\newcommand{\R}{\mathbb R}%
\newcommand{\C}{\mathbb C}%
\renewcommand{\H}{\mathbb H}%

\newtheorem{defn}{Definition}[section]
\newtheorem{thm}[defn]{Theorem}
\newtheorem{lem}[defn]{Lemma}
\newtheorem{prop}[defn]{Proposition}
\newtheorem{cor}[defn]{Corollary}

\newcommand{\be}{\begin{equation}}
\newcommand{\ee}{\end{equation}}
\begin{document}
\title[Immersions in Quaternionic Grassmannians]{Immersions in a Quaternionic Grassmannian\\ inducing a given 4-form}
\author[M. Datta]{Mahuya Datta}
\address{Statistics and Mathematics Unit, Indian Statistical Institute\\ 203,
B.T. Road, Calcutta 700108, India.\\ e-mail:
mahuya@isical.ac.in}
\keywords{Connections, Symplectic Pontrjagin forms, Quaternionic Grassmannians, Immersions}
\thanks{2010 Mathematics Subject Classification:  53C07, 57R42, 58A10, 58B05, 58J99.}
\begin{abstract} Let $Gr_k(\H^n)$ be the Grassmannian manifold of Quaternionic $k$-planes in $\H^n$ and let $\gamma^n_k\to Gr_k(\H^n)$ denote the Stiefel bundle of quaternionic $k$-frames in $\H^n$. Let $\sigma$ denote the first symplectic Pontrjagin form associated with the universal connection on $\gamma^n_k$. We show that every 4-form $\omega$ on a smooth manifold $M$ can be induced from $\sigma$ by a smooth immersion $f:M\to Gr_k(\H^n)$ (for sufficiently large $k$ and $n$) provided there exists a continuous map $f_0:M\to Gr_k(\H^n)$ which pulls back the cohomology class of $\sigma$ onto that of $\omega$.
\end{abstract}
\maketitle
\section{Introduction}
In \cite{datta2} we proved that the complex Grassmannians $Gr_k(\C^n)$ admit some even degree differential forms $\sigma_i$ of degree $2i$ which are universal. This means that any closed differential $2i$-form $\omega$ on a manifold $M$ can be obtained as the pullback of $\sigma_i$ by an immersion $f:M\to Gr_k(\C^n)$ (for sufficiently large $n$) provided there is a continuous map $f_0:M\to Gr_k(\C^n)$ which pulls back the deRham cohomology class of $\sigma_i$ onto that of $\omega$. The immersion obtained in this case is homotopic to $f_0$. The result for 2-forms were known for some time.  Tishcler and Gromov proved independently that the complex projective spaces with the first Chern forms are the universal object. Our theorem genralised these results for even degree forms.

In this paper we prove that there exists a 4-form on the quaternionic Grassmannian $Gr_k(\H^n)$ which is also universal in the above sense. This 4-form is defined as the first symplectic Pontrjagin form $p$ of the universal connection on the Stiefel bundle $\gamma^n_k\to Gr_k(\H^n)$. The main result may be stated as follows.
\begin{thm}Let $M$ be a closed manifold of dimension $m$ and let $\omega$ be a closed 4-form on it. Suppose that there exists a continuous map $f_0:M\to Gr_k(\H^n)$ which pulls back the cohomology class of $p$ onto that of $\omega$. Then there exists a smooth immersion $f:M\to Gr_k(\H^n)$ such that $f^*p=\omega$, provided $k\geq$ and $n\geq $.\label{main}
\end{thm}
On our way to achieving this goal we also prove the following result.
\begin{thm} Let $P$ be an $Sp(k)$-bundle over a closed manifold $M$ and $\omega$ a 4-form on $M$ representing the first Symplectic Pontrjagin class of $P$. Then there is a connection $\alpha$ on $P$ such that the symplectic Pontrjagin form of $\alpha$ is $\omega$, provided $k\geq $.\label{primary}
\end{thm}


The Stiefel bundle $\gamma^n_k\to Gr_k(\H^n)$ is a universal $Sp(k)$-bundle which admits a universal connection \cite{schlafly}. This helps us to show that the statements of Theorem~\ref{main} and Theorem~\ref{primary} are equivalent.
In order to prove Theorem~\ref{primary} we note that if $k>m=\dim M$, then any $Sp(k)$-bundle $P$ can be reduced to a $Sp(m)\times Sp(1)^{k-m}$ bundle $P_1\times P_2$, where $P_2$ is a trivial $Sp(1)^{k-m}$ bundle. Since, the symplectic Pontrjagin form is additive the problem reduces to showing that every exact form on $M$ can be realised as the first symplectic Pontrjagin form of some connection on the trivial $Sp(1)^{k-m}$ bundle $P_2$. This is equivalent to showing a decomposition of an exact 4-form as the sum of squares $d\beta_i\wedge d\beta_i$, $i=1,2,\dots,k-m$, where each $\beta_i$ is a 1-form on $M$. In order to obtain the decomposition we consider a differential operator $\bar{\mathcal D}$ which we show to be infinitesimally invertible. We then apply an Implicit function theorem for smooth differential operators on $\bar{\mathcal D}$ to show that it is surjective. 


We also prove the following:

\begin{cor} Let $M$ be any manifold. Every 3-form on $M$ is the secondary characteristic form of a pair of connections $(\omega_0,\omega)$ on a trivial $Sp$ bundle, modulo exact forms, where $\omega_0$ is the trivial connection.

Every exact 4-form on $M$ is the symplectic Pontrjagin form of some connection on a trivial $Sp$ bundle over $M$.\end{cor}


Analogues of Theorem~\ref{secondary} and Corollary~\ref{primary} were proved in \cite{datta2} for homogeneous components $ch_k$ of the Chern character $ch=\sum_{k\geq 0}\frac{1}{k!}ch_k$ of a principal $U(n)$-bundle, where $ch_k\in H_{deR}^{2k}(M, \R)$. Explicitly, we showed that if $P$ is a principal $U(n)$ bundle, then every $2k$-form representing the cohomology class $ch_k(P)$ is a Chern character form of some connection on the bundle provided $n$ is large. In other words, every form representing the characteristic class $ch_k(P)$ is a Chern character form of the stable isomorphism class $\{P\}$. The principal idea was to use the additivity property of the Chern character forms. A careful observation of the proof also shows, that every $(2k+1)$-form on $M$, modulo exact forms, can be realised as a Chern-Simons form of a pair of connections $(\omega_0,\omega)$ on $P$, for some fixed $\omega_0$. However, this is not explicitly mentioned in the article. In a recent paper\cite{sullivan}, Sullivan and Simons have shown that any odd degree form (not necessarily homogeneous) is a Chern-Simons form of some unitary connection on a trivial bundle by exploiting the multiplicative property of the Chern character forms and the partial additivity of the Chern-Simons classes. This gives a differential $K$-theory of unitary bundles. We also refer to \cite{pingali} for a similar result.

Let  $G$ be one the Lie groups $O(n), U(n)$ or $Sp(n)$. We say that $\alpha$ is a characteristic form of a principal $G$-bundle $P$, if it is the characteristic form of some connection on $P$; $\alpha$ will be called a characteristic form of a stable isomorphism class $\{P\}$ if it is a characteristics form of some $Q\in\{P\}$.

\section{Preliminaries}
Let $\H$ denote the (skew-)field of quaternions represented by elements of the form $u+x i+y j+z k$, where $u,x,y,z\in\R$ and $i,j,k$ satisfy the relations
\[i^2=j^2=k^2=-1 \ \  ij=k=-ji,\ jk=i=-kj,\ ki=j=-ik.\]
$\H$ becomes a division ring with respect to coordinatewise addition and the above multiplication. To each element $w=u+x i+y j+z k$, we can associate a conjugate $\bar{w}$ which is of the form $u-x i-y j-z k$.
The $n$-fold product of $\H$, can be given the structure of a quaternionic right (resp. left) vector space of dimension $n$ and comes with
the following canonical inner product on $\H^n$: For $w,w'\in\H^n$, the inner product $\langle w,w'\rangle=\sum_{i=1}^nw_i\bar{w'_i}$, where $w_i$, $w_i'$ denote the coordinates of $w$ and $w'$ respectively. The subgroup of $Gl_n(\H)$ consisting of all linear transformations which preserve the inner product is denoted by $Sp(n)$. This is called the Symplectic group. The quaternions can be realised as a 2-dimensional vector space over the complex numbers with basis $\{1,j\}$ by identifying the element $u+x i+y j+z k$ with the ordered tuple $(u+xi, y+zi)$.

With this identification we can identify a quaternionic vector space $V$ of dimension $n$ with complex vector space of dimension $2n$. Then the quaternionic inner product of two elements of $v,w\in V$, can be expressed as follows:
\[\langle v,w\rangle=h(v,w)+j\omega(v,w)\]
where $h(v,w)$ denotes the canonical Hermitian inner-product on $\C^{2n}$ and $\omega$ denotes the complex symplectic form on $\C^{2n}$.
This shows that the Symplectic group can also be seen as a subgroup of $U(2n)$ consisting of all $A\in U(2n)$ which preserve the complex symplectic form $\omega$. Thus $Sp(n)=U(2n)\cap Sp(2n,\C)$ and $X\in Sp(n)$ if and only if 
\[XX^*=I=X^*X \ \ \ \mbox{ and } \ \ \  X^tJX=J,\] 
where $X^*$ denote the adjoint of $X$ and $J$ is of the form \[\left(\begin{array}{cc}0 & I_n\\-I_n & 0\end{array}\right)\]
Therefore the Lie algebra ${\mathfrak sp}_n$ of $Sp(n)$ can be identified with the subalgebra of all skew-Hermitian $2n\times 2n$ matrices satisfying $Y^tJ+JY=0$.

\subsection{Universal connections on Stiefel bundles} Let $Gr_k(\H^n)=Sp(n)/Sp(k)\times Sp(n-k)$ be the Grassmannian manifold of quaternionic $k$-subspaces in $\H^n$ and $p:\gamma^n_k\to G_k(\H^n)$ denote the Stiefel bundle of orthonormal $k$-frames in $\H^n$. This is a principal $Sp(k)$ bundle over $G_k(\H^n)$. The projection map $p$ maps a $k$-frame onto the subspace of $\H^n$ spanned by these vectors.
$G_k(\H^n)$ is the classifying space of all principal $Sp(k)$-bundles over manifolds of dimension $m$, provided $n\geq $. This means that every principal $Sp(k)$ bundle over a manifold of dimension $m$ can be obtained as a pullback bundle of $\gamma^n_k$ via some map $f:M\to G_k(\H^n)$. Indeed, there is a one-to-one bijection from the set homotopy classes of continuous maps $[M, G_k(\H^n)]\to {\mathcal V}^k(M)$, where $\mathcal V^k(M)$ denote the set of isomorphism classes of principal $Sp(k)$ bundles over $M$, provided $n$ is sufficiently large.
Following the work of Narasimhan and Ramanan \cite{nara} on unitary bundles, Schlafly constructs a connection $\omega_0$ on $\gamma^n_k$,  and proves that every connection on a principal $Sp(k)$ bundle over a manifold $M$ of dimension $m$ can be induced from $\omega_0$ by some map $f:M\to Gr_k(\H^n)$, provided $n\geq k(m+1)(4mk^2+2mk+1)$. We shall refer to this connection $\omega_0$ as universal connection on $\gamma^n_k$.

\subsection{Symplectic Pontrjagin Classes}
It is a general fact that two principal bundles over $M$ which are stably isomorphic, have the same characteristic classes. The characteristic classes, we may recall, are certain cohomology classes in $H^*(M,\R)$ which are invariants of vector bundles. The geometric theory of characteristic classes of a principal $G$-bundle $P\to M$ associates a cohomology class to an invariant polynomial on the lie algebra $\mathfrak g$ of the Lie group $G$.  Evaluating the polynomial on the curvature forms of connections on $P$ we get the representatives of the characteristic class. Each such form is referred as a characteristic form of the bundle $P$. A differential form $\omega$ will be called a characteristic form of a stable isomorphism class $\{P\}$ if it is the characteristic form of some $Q$ in the stable isomorphic class of $P$.

We shall now describe the Symplectic Pontrjagin classes of principal $Sp(k)$ bundles. Given any $Sp(k)$ bundle $P\to M$ we can associate to it some cohomology classes lying in $H^{4i}(M,\R)$ which are topological invariants of the principal bundle. These are called symplectic Pontrjagin classes. The Chern Weil theory says that each $Sp(k)$-invariant polynomial $f:{\mathfrak sp}(k)\to \R$ of even degree $2i$ defines a cohomology class in $H^{4i}(M,\R)$ which is an invariant of the bundle. The generators of the ring of all invariant polynomials can be obtained as follows.
For $X\in{\mathfrak sp}(k)$, we expand $\det(\lambda I_{2k}+iX)$ to write it as a polynomial in $\lambda$:
\[\lambda^{2k}-f_1(X)\lambda^{2k-2}+f_2(X)\lambda^{2k-4}+\dots\]
First note that there can not be any odd degree terms in the above expression because of the following simple reason. Since $X$ is a Lie algebra element it satisfies the relation $JX+X^tJ=0$. Therefore,
\[\begin{array}{rcl}\det(\lambda I_{2k}+iX)& = & \det(\lambda I_{2k}+iX)^t\\
& = & \det(\lambda I_{2k}+iX^t)\\
& = & \det(\lambda I_{2k}-iJXJ^{-1})\\
& = & \det(\lambda I_{2k}-iX).\end{array}\]
Consequently, the coefficients of odd powers of $\lambda$ vanish.
It can be proved that the polynomials $f_1, f_2, \dots$ generate the space of all invariant polynomials.

Let $P\to M$ be a principal $Sp(k)$-bundle, and let Ad\,$P$ denote the vector bundle $P\times_G{\mathfrak g}$ over $M$, where $G$ action on $\mathfrak g$ is the adjoint action. The space of $Sp(k)$-connections on $P$ is an affine space which may be parametrised by the space of Ad\,$P$-valued 1-forms on $M$. This means that if we fix a connection $\omega$ then any connection can be expressed as $\omega+\alpha$ for some section  $\alpha$ of $T^*(M)\otimes \mbox{Ad}\,(P)$ and conversely. The curvature form $D\omega$ of $\omega$ is defined as $D\omega=d\omega+\omega\wedge\omega$. We shall often denote the curvature by $\Omega$. Now consider an $Sp(k)$-connection $\omega$ on the principal bundle $P$ and substitute its curvature form $D\omega$ in the polynomial $f_i$ for $X$. Then, it gives a $4i$-form on $P$ which projects onto a closed $4i$-form on $M$. We shall denote this form by $p_i(\omega)$ and call it the $i$-th symplectic Pontrjagin form of $\omega$. The cohomology class of $p_i(\omega)$ is independent of the choice of the connection. We shall denote the cohomology class by $p_i$ and we refer to it as the symplectic Pontrjagin class the bundle $P$. 

We now restrict our attention to only the first symplectic Pontrjagin form $p_1$.
It may be shown that \[p_1(\omega+\alpha)=p_1(\omega)+d\int_0^1\mbox{trace }\alpha\wedge D(\omega+t\alpha) dt.\]
The 3-form on $P$ defined by the integral $\int_0^1 \mbox{trace }\alpha\wedge D(\omega+t\alpha)\,dt$ projects onto $M$. We shall denote the projected form on $M$ by $\bar{p}_1(\omega, \omega+\alpha)$ and will refer to it as the secondary characteristic form. While the primary characteristic forms are only topological invariant of the bundles, the secondary characteristic forms contain important geometric information. In particular, if we consider the trivial principal bundle and take $\omega$ to be the trivial connection on it, then $\bar{p}_1(0,\alpha)=\frac{1}{2} \mbox{trace }(\alpha\wedge D\alpha)$.

We shall now show that the functor $p_1$ is additive. This means that if  $\xi$ is an $Sp(k)$ bundle and $\xi'$ an $Sp(k')$ bundle,  and $\omega$, $\omega'$ are two connections on $\xi$, $\xi'$ respectively then $p_1(\omega\oplus\omega')=p_1(\omega)+p_1(\omega')$, where $\omega\oplus\omega'$ is the induced connection on the $Sp(k+k')$-bundle $\xi\oplus\xi'$. This observation is crucial for our result (and was pointed out to me by M. Gromov). Using the description of $f_1$ we can write
\[\begin{array}{rcl}f_1(X) & = & -\sum \mbox{ principal 2-minors }\\
                           & = & - \sum_{i< j} (x_{ii} x_{jj}-x_{ij}x_{ji})\\
                           & = & (\sum_{i< j} x_{ij}x_{ji}+ \frac{1}{2}\sum x_{ii}^2)- (\sum_{i< j} x_{ii} x_{jj} + \frac{1}{2}\sum x_{ii}^2)\\
                           & = & \frac{1}{2}\left(\mbox{trace }X^2-(\mbox{trace }X)^2\right)
                           \end{array}\]
Now note that a maximal torus $\mathfrak t$ of $Sp(k)$ consists of diagonal matrices of the form
\[\left(\begin{array}{cccccc}
                         i\xi_1 &  &  &  &  &  \\
                            &  \dots &  &  &  &  \\
                            &  & i\xi_n &  &  &  \\
                            &  &  & -i\xi_n &  &  \\
                            &  &  &  &  \dots &  \\
                            &  &  &  &  & -i\xi_n \\
                       \end{array}
                     \right)\]
where $\xi_i$'s are real numbers. Therefore, if we restrict the invariant polynomial $\mbox{trace }X$ to the torus then it is identically zero. Hence it is identically zero on ${\mathfrak sp}(k)$ and so $f_1(X)=\mbox{trace }X^2$ for all $X\in{\mathfrak sp}(k)$. Consequently,
$p_1(\omega\oplus\omega')=f_1(\Omega\oplus\Omega')=f_1(\Omega\wedge\Omega)+f_1(\Omega'\wedge\Omega')=p_1(\omega)+p_1(\omega')$. This proves that $p_1$ is additive.

Let $P\to M$ be a trivial $Sp(q)$ bundle over a manifold $M$. A connection on $P$ is a $\mathfrak{sp}(q)$-valued 1-form on $M$.  If we consider a connection $\omega$ on $P$ which is a diagonal matrix with diagonal entries  $i\omega_1,\dots,i\omega_q,-i\omega_1,\dots,-i\omega_q$
then $p_1(\omega)=_{i=1}^q\sum(d\omega_i)^2$. Thus every exact form which can be expressed as the sum of squares of exact 2-forms is the first symplectic Pontrjagin form of some connection.

\section{Connections with prescribed characteristic forms}

The results of this section were proved earlier in \cite{datta1} and \cite{datta2}. For the sake of completeness we present the relevant part from there.

We first prove that every differential 4-form on $M$ can be expressed as the sum of the squares of exact forms. In order to see this we consider a differential operator $\mathcal D$ which takes a $q$-tuple of $1$-forms $(\omega_1,\dots,\omega_q)$ onto $\sum_{i=1}^q (d\omega_i)^2$.
We want to show that the image of this operator consists of all exact 4-forms. Observe that if \be \sum_{i=1}^q(d\omega_i)^2=d\alpha \label{decomposition}\ee
then $\alpha-\sum_{i=1}^q\omega_i\wedge d\omega_i$ is a closed form and conversely. Hence every tuple $(\omega_1,\omega_2,\dots,\omega_q)$ for which $\alpha-\sum_{i=1}^q\omega_i\wedge d\omega_i$ is exact certainly satisfies equation (\ref{decomposition}).

There is an associated operator $\bar{\mathcal D}$ defined as follows:
\[\begin{array}{rcl}\bar{\mathcal D}:\Omega^1(M)^q\times \Omega^2(M) & \longrightarrow & \Omega^3(M)\\
(\omega_1,\dots,\omega_q), \phi & \mapsto &  \sum_{i=1}^q\omega_i\wedge d\omega_i + d\phi\end{array}\]

\begin{defn} A $q$-tuple of 1-forms $(\omega_1,\dots,\omega_q)$ is said to be regular if the linear map $L:(\Omega^1(M))^q\to \Omega^3(M)$ defined by
\[(\alpha_1,\dots,\alpha_q)\mapsto \sum_{i=1}^q \alpha_i\wedge d\omega_i\]
is an epimorphism.\end{defn}
\begin{itemize}\item
If $d\omega_1,\dots,d\omega_q$ span the exterior bundle $\Lambda^2(M)$ then $(\omega_1,\dots,\omega_q)$ is a regular $q$-tuple. 
\item The operator $\bar{\mathcal D}$ is infinitesimally invertible at a regular $q$-tuple \cite{datta1}. 
\end{itemize} 
The regularity being an open condition, the set of regular $q$-tuples form an open subset in the fine $C^\infty$ topology.  We would like to show that this set is non-empty for large $q$.

\begin{prop} If $q\geq m(m+1)/2$, then there exist smooth functions $f_i:M\to\R^2$, $i=1,2,\dots,q$ such that the 2-forms $f_i^*(dx\wedge dy)$ span the exterior bundle $\Lambda^2(T^*M)$. \label{large}
\end{prop}
We first prove an algebraic result.\footnote{The above proposition is a particular case of a result that appeared in \cite{datta1}. The proof requires the algebraic lemma which was overlooked previously.}
\begin{lem}
Let $L^j$, $j=1,2,\dots,q$, be $n\times 2$ matrices whose coordinates are indeterminates. Let the $2\times 2$ minors of $L^j$ be denoted by $p^i(L^j)$, $i=1,2,\dots,N={n\choose 2}$. Form a matrix whose entries are $p^i(L^j)$, $i=1,2,\dots,N$ and $j=1,2,\dots,q$. Then the determinant of any $N\times N$ minor is an irreducible polynomial. \end{lem}

\begin{proof}Denote the row vectors of $L^j$ by $(x^j_k, \bar{x}^j_k)$. Without loss of generality consider the first $N$ columns of the matrix and call the matrix formed by these columns by $Y$. We will show by finite induction that $\det Y$ is an irreducible polynomial. The matrix $Y$ has the following two important properties:\begin{enumerate}
\item[(a)] Every entry of the matrix is an irreducible polynomial because it is the determinant of a matrix who entries are independents variables. 
\item[(b)] The determinant of $Y$ is a homogeneous polynomial which is linear in every indeterminate.
\end{enumerate}
The first observation starts the induction. We shall now show that if all $k\times k$ principal minors of $Y$ are irreducible polynomials then so are the  $(k+1)\times (k+1)$ minors. To see this it enough to consider any one of the $(k+1)\times(k+1)$ minor and we take the first $(k+1)\times (k+1)$-principal minor $Q_{k+1}= \left(p^i(L^j)\right)$, $i,j=1,2,\dots,{k+1}$. This also has the two properties mentioned above. If we expand the determinant with respect to the last column it can be written in the form 
$\sum _{i=1}^n x^{k+1}_iB_i$, where $B_i=\sum_{m=1}^n \bar{x}^{k+1}_m A_{im}$ and $A_{im}$ are $k\times k$ minor of the $(k+1)\times (k+1)$ matrix corresponding to the elements of the last column. None of the variables $x^{k+1}_i, \bar{x}^{k+1}_i$, $i=1,2,\dots,n$ appear in $A_{im}$. Since each $A_{im}$ is irreducible and $B_i$ is a linear polynomial it follows that $B_i$ is irreducible for each $i$. By the same argument again, $Q_{k+1}$ is irreducible.
\end{proof}

{\em Proof of Propostion 3.1: }
Fix a basis $e_1,e_2,\ldots,e_m$ for $\R^m$.
Let $L:\R^m \to \R^2\times\dots\times\R^2$ be a linear map into the $q$-fold product of $\R^2$. Then $L$ can be expressed as
$L=(L_1,L_2,\dots,L_q)$, where $L_i$ is the projection of $L$ onto the $i$-th copy of $\R^2$.

Denote the canonical volume form on $\R^2$ by $\sigma$. Suppose that $L_1^*\sigma, 
L_2^*\sigma,\dots, L_q^*\sigma$ span the bundle $\Lambda^2(\R^m)$. The
$2\times 2$ cofactors of $L_i$ correspond to the
values of $L_i^*\sigma$ on the $2$-tuples of basis vectors $(e_{i_1},e_{i_2})$, where $\{i_1,i_2\}$ is an ordered subset of $\{1,2,\dots,m\}$.  If
$\bar{L}_i$ denote the column vector formed by the $2\times 2$
cofactors of the matrix $L_i$ then the above condition means
that $\bar{L}=(\bar{L}_1,\dots,\bar{L}_q)$ has the maximum rank. Let
$\Sigma'$ consist of all linear maps $L=(L_1,\dots,L_q):\R^m\longrightarrow\R^{2q}$ such that
rank\,$\bar{L}$ is strictly less than $l={m\choose 2}$; in other words,
any $l\times l$ cofactor of $\bar{L}$ is zero. Consider all $l\times l$ cofactors which contains the first $l-1$ columns:
\[P_i=\det (\bar{L}_1\ \bar{L}_2\ \dots \bar{L}_{l-1}\ \bar{L}_k), \ \ \ \ \ k\in\{l,l+1,\dots,q\}.\]
By the lemma above, each polynomial $P_i$ is irreducible. Moreover, the polynomials are clearly independent, as the variables of $P_i$ is not contained in the variables appearing in the rest of the polynomials. Thus the common zero set of the polynomials $P_i$'s has the maximum codimension. This implies that $\Sigma'$ is semialgebraic and hence stratified (\cite{gromov}, 1.3.1) and the codimension of $\Sigma'$ in $L({\mathbb R}^m,{\mathbb R}^{2q})$ is $q-{m\choose 2}+1$.

Let $\Sigma$ be the subset of the 1-jet space $J^1(M,N)$
consisting of all 1-jets $j^1_f(x)$ such that $\{f_i^*\sigma:\
i=1,2,\dots,q\}$ do not span $\Lambda^k_x(M)$. Hence a map
$f:M\longrightarrow N$ is $\sigma$-large if its 1-jet map misses
the set $\Sigma$. Since $\sigma$ has global symmetry, the singular
set $\Sigma$ in the 1-jet space fibres over $M$ and therefore it
is stratified with codimension $q-{m\choose 2}+1$. Hence by the Thom
Transversality Theorem, a generic map is $\sigma$-large if $q-
{m\choose 2}\geq m$.\qed
\begin{prop}
There exists a regular $q$-tuple $(\omega_1,\dots,\omega_q)$ such that $\sum_{i=1}^q\omega_i\wedge d\omega_i=0$.\label{decomposition0}
\end{prop}
\begin{proof} Let $f=(f_1,\dots,f_q)$ be as in Proposition~\ref{large}. Note that $\sigma=dx\wedge dy$ is an exact form and so we can write it as $\sigma=d\tau$. Set $\omega_i=f_i^*\tau$. By our assumption $(\omega_1,\dots,\omega_q)$ is regular. Further, since any 3-form on $\R^2$ is zero, $\omega_i\wedge d\omega_i=f_i^*(\tau\wedge \sigma)=0$. Therefore, 
$\sum_{i=1}^q\omega_i\wedge d\omega_i=0$. This proves the proposition.
\end{proof}

\begin{prop} Let $M$ be a closed manifold of dimension $m$. If  $q\geq m(m+1)/2$, then 
\begin{enumerate}\item every 3-form on $M$ can be expressed as \begin{center}$\sum_{i=1}^q\omega_i\wedge d\omega_i+$ an exact form\end{center} and hence  
\item every exact $4$-form $\sigma$ on $M$ can be expressed as 
\[\sigma=\sum_{i=1}^q(d\omega_i)^2,\]
where $\omega_i$, $i=1,\dots,q$, are 1-forms on $M$.\end{enumerate}
\label{decomposition1}\end{prop}
\begin{proof}
Let $S$ denote the set of regular $q$-tuples of 1-forms on $M$. Since $\bar{\mathcal D}$ is  infinitesimally invertible on $S$, its restriction to this set is an open map. Now, $S$ being a non-empty open set, the image of $\bar{\mathcal D}$ is non-empty and open. We shall now show that $\bar{\mathcal D}$ is surjective.
\end{proof}
\section{Proof of the main results}
In this section we prove the main results of the article.
\begin{prop} Let $P\to M$ be a principal $Sp(k_0)$ bundle over a closed manifold $M$. Let $\omega_0$ be a connection on $P$ and let $\alpha$ be any differential 3-form on $M$. There exists a connection $\omega$ on the trivial $Sp(1)^k$-bundle $P_0$ such that the secondary characteristic form of the pair $(\omega_0\oplus 0,\omega_0\oplus \omega)$ on $P\oplus P_0$ is equivalent to $\alpha$ modulo an exact 3-form, provided $k\geq \frac{m(m+1)}{4}$.\label{secondary}
\end{prop}
\begin{proof}The secondary characteristic class for a pair of connections $(\omega_0\oplus 0,\omega_0\oplus\omega)$ as in the statement above can be expressed as 
\[\begin{array}{rcl}\bar{p}_1(\omega_0\oplus 0,\omega_0\oplus\omega) & = & \int_0^1\text{trace}((0\oplus\omega)\wedge D(\omega_0\oplus t\omega))\\
& = & \int_0^1 \text{trace }(t\omega\wedge D\omega)\,dt\\
\end{array}\] 
If we take $\omega$ in the $2\times 2$ block diagonal form with diagonal entries 
\[\left(
  \begin{array}{cc}
    i\alpha_j & 0 \\
    0 & -i\alpha_j
  \end{array}
\right)\]
where $\alpha_j\in \Omega^1(M)$ for $j=1,2,\dots,k$, then 
\[\bar{p}_1(\omega_0\oplus 0,\omega_0\oplus\omega)=\sum_{j=1}^k\alpha_j\wedge d\alpha_j.\]
The result follows in view of Propositions~\ref{decomposition1}.
\end{proof}
Let $P$ be a principal $Sp(n)$ bundle and let $\langle P\rangle$ denote its stable equivalence class. By a characteristic form of the stable equivalence class of $P$ we shall understand a form which is the characteristic form of some connection on a principal bundle $Q\in\langle P\rangle$. Observing that every principal bundle $P\oplus Sp(k)$ represents $\langle P\rangle$ we can conclude from the above result that every 3-form on $M$, modulo some exact form, is the secondary characteristic form of $\langle P\rangle$.
Since the characteristic classes of a principal bundle only depends on its stable equivalence class we conclude from the 
We also prove the following:

Let $M$ be any manifold. Every 3-form on $M$ is the secondary characteristic form of a pair of connections $(\omega_0,\omega)$ on a trivial $Sp$ bundle, modulo exact forms, where $\omega_0$ is the trivial connection.
\begin{prop} 
Let $M$ be an $m$-dimensional manifold and let $n$ be an integer satisfying $n\geq \frac{m(m+1)}{6}$. Then every exact 4-form on $M$ is the first symplectic Pontrjagin form of some connection on a trivial $Sp(n)$ bundle over $M$.\label{exact_char}\end{prop}
\begin{proof} If $P$ is a trivial $Sp(1)$-bundle then a connection is simply a ${\mathfrak sp}(1)$-valued 1-form $M$. Therefore, a connection $\omega$ on $P$ is of the following form:
\[\left(
  \begin{array}{cc}
    i\alpha & \beta+i\gamma \\
    -\beta+i\gamma & -i\alpha
  \end{array}
\right)\]
where $\alpha$, $\beta$ and $\gamma$ are smooth 1-forms on $M$. The curvature form of $\Omega$ is equal to $d\omega$ since $\omega\wedge\omega=0$, and  $p_1(\omega)= \mbox{\,trace\,}\Omega^2=(d\alpha)^2+(d\beta)^2+(d\gamma)^2$. Consider a $Sp_k$-connection $\omega$ on the trivial bundle over $M$ which is of the form Diag $(\omega_1,\omega_2,\dots,\omega_k)$, where each $\omega_i$ is a $2\times 2$ block of the form described as above. Thus every 3-form on $M$ can be expressed as the Symplectic Pontrjagin form of some connection on trivial $Sp(k)$ bundle for $3k\geq \frac{m(m+1)}{2}$.
\end{proof}

\begin{thm}  Let $P$ be a principal $Sp(n)$ bundle over a
closed manifold $M$ of dimension $m$. Then every $4$-form representing the
the first Symplectic Pontrjagin class $p_1$ of $P$ is the Symplectic Pontrjagin form
of some connection on $P$ provided 
\[n>m_0+ m(m+1)/6.\] 
where $m_0=[(m+1)/4]-1$.\end{thm}
\begin{proof} If $n>m_0$ then $P$ can be
reduced to $P_1\oplus P_2$, where $P_1$ is a principal $Sp(m_0)$
bundle and $P_2$ is the trivial $Sp(n-m_0)$ bundle over $M$ \cite{husemoller}. Moreover, we
have a canonical inclusion $Q=P_1\oplus
P_2\stackrel{i}{\longrightarrow}P$ which takes fibres of
$P_1\oplus P_2$ canonically into the fibres of $P$. It is a standard fact that a connection $\alpha_Q$ on $Q$
can be extended uniquely to a connection $\alpha_P$ on $P$ such
that $i^*\alpha_P=\alpha_Q$. We shall show that
$p_1(\alpha_Q)=p_1(\alpha_P)$. We recall that the Symplectic Pontrjagin 
form $p_1(\alpha_Q)$ is uniquely determined by the
equation
\begin{equation}\pi^*_Q\,p_1(\alpha_Q)=
\mbox{trace\,}(D\alpha_Q)^2,\label{unique}\end{equation} where $D$
stands for the covariant differentiation and $\pi_Q$ denotes the
projection map $Q\longrightarrow M$. Similarly,
$\pi^*_P\,p_1(\alpha_P)=\mbox{trace\,}(D\alpha_P)^2$
(\cite{kobayashi}). Taking the pull back by $i$ we get
$i^*\pi^*_P\,p_1(\alpha_P)=\mbox{trace\,}(D\alpha_Q)^2$. Since
$\pi_P\circ i=\pi_Q$, the left hand side is equal to
$\pi^*_Q\,p_1(\alpha_P)$. Hence by equation (~\ref{unique}) and
the uniqueness property $p_1(\alpha_P)=p_1(\alpha_Q)$. Now,
the first Symplectic Pontrjagin form being additive, if $\alpha_1$ and
$\alpha_2$ are connections on $P_1$ and $P_2$ respectively then
$p_1(\alpha_1\oplus\alpha_2)=p_1(\alpha_1)+p_1(\alpha_2)$. In
view of this observation it is enough to show that every
exact form on $M$ is the first Symplectic Pontrjagin 
form of some connection on the trivial principal $Sp(n-m_0)$-bundle $P_2$.
Now by Proposition~\ref{exact_char}, every exact $4$-form is the first Symplectic Pontrjagin form of some connection on $P_2$ provided $n-m_0>\frac{m(m+1)}{6}$.
\end{proof}

Proof of Theorem 1.1.
\begin{proof} Without loss of generality we may assume that $f_0$
is smooth. Consider the pull-back bundle $f_0^*(\gamma^q_n)$ and
call it $P$. By the naturality of the Symplectic Pontrjagin classes, the
first Symplectic Pontrjagin class of $P$ is $[\sigma]$. Therefore,
by the above theorem, there exists a connection $\alpha$ on
$P$ satisfying $p_1(\alpha)=\sigma$, provided $n\geq m_0+\frac{m(m+1)}{6}$.

On the other hand, if $q\geq n(m+1)(4mn^2+2mn+1)$ then there exists a bundle map
$(F,f):P\lgra \gamma^q_n$ such that $F^*\alpha_0=\alpha$. Then
$p_1(\alpha)=p_1(F^*\alpha_0)=f^*p_1(\alpha_0)$ and hence
$\sigma=f^*p_1(\alpha_0)$. Moreover, since $P$ is isomorphic to
$f^*(\gamma^q_n)$, $f:M\lgra Gr_n(\H^q)$ is homotopic $f_0$. This
proves the theorem.\end{proof}

{\em Acknowledgements}. I would like to thank Misha Gromov for his comment on the symplectic Pontrjagin form. I would also like to thank Amartya Dutta and Sudhir Ghorpade for their suggestions and discussion on the problem.


\begin{thebibliography}{99}
\bibitem{datta1} Datta, M. Connections with prescribed first Pontrjagin form. Trans. Amer. Math. Soc. \textbf{355} (2003), no. 9, 3813 -- 3824
\bibitem{datta2} Datta, M. Universal property of Chern character forms of the canonical connection. \emph{Geom. Funct. Anal.} \textbf{14} (2004), no. 6, 1219 -- 1237.
\bibitem{gromov} Gromov, Mikhael \emph{Partial differential relations}. Ergebnisse der Mathematik und ihrer Grenzgebiete \textbf{9}. Springer-Verlag, Berlin, 1986.
\bibitem{husemoller}Husemoller, D. \textit{Fibre bundles}. Third edition. Graduate Texts in Mathematics, \textbf{20}. Springer-Verlag, New York, 1994. xx+353 pp.
\bibitem{kobayashi}  Kobayashi, S.; Nomizu, K. \textit{Foundations of differential geometry}. Interscience Tracts in Pure and Applied Mathematics, No. 15 Vol. II. Interscience Publishers John Wiley \& Sons, Inc., New York-London-Sydney 1969 xv+470 pp. 
\bibitem{nara} Narasimhan, M. S.; Ramanan, S. Existence of universal connections. \emph{Amer. J. Math.} \textbf{83} 1961 563 -- 572.
\bibitem{pingali} Pingali, V. P. and Takhtajan, Leon A. On Bott-Chern forms with applications to differential K-theory. arXiv:1102.1105v2 [math.DG] 28 Mar 2011
\bibitem{schlafly}Schlafly, R. Universal connections: the local problem. \emph{Pacific J. Math}. \textbf{98} (1982), no. 1, 157 -- 171.
\bibitem{sullivan} Simons, J.; Sullivan, D. \textit{Axiomatic characterization of ordinary differential cohomology}, J. Topol. 1 (2008), no. 1, 45–56.
\bibitem{tischler}Tischler  D.  Closed  2-forms  and an embedding theorem for symplectic manifolds,  \emph{J. Differential  Geom.}  \textbf{12}, 229 -- 235  (1977)
\end{thebibliography}
\end{document}